\documentclass[leqno]{article}
\usepackage[cp1250]{inputenc}
\usepackage{authblk}
\usepackage[english]{babel}

\usepackage{a4wide}
 
\usepackage{amsfonts}
\usepackage{graphicx}
\usepackage{epstopdf}
\usepackage{bm}
\usepackage{subfig}
\usepackage{amsthm} 
\usepackage{amsmath}
\usepackage{adjustbox}

\newcommand{\grad}{\nabla}

\newcommand{\dx}{\ d \boldsymbol{x}}

\newcommand{\ds}{\ d s}

\newtheorem{lemma}{Lemma}
 
\begin{document}

\title{\textbf{Stabilised hybrid discontinuous Galerkin methods for the Stokes problem with non-standard boundary conditions
}} \author[1]{Gabriel R. Barrenechea \thanks{E-mail: gabriel.barrenechea@strath.ac.uk}}
\author[1,2]{Micha\l{} Bosy \thanks{E-mail: michal.bosy@unipv.it}}
\author[1,3]{Victorita Dolean \thanks{E-mail: victorita.dolean@strath.ac.uk}}
\affil[1]{\small \textit{Department of Mathematics and Statistics, University of Strathclyde, 26 Richmond Street, G1 1XH Glasgow, United Kingdom}}
\affil[2]{\small \textit{Dipartimento di Matematica ''F. Casorati'', Universit\'{a} degli Studi di Pavia, Via Adolfo Ferrata 5, 27100 Pavia, Italy}}
\affil[3]{\small \textit{University C\^ote d'Azur, CNRS, LJAD, 06108 Nice Cedex 02, France}}

\date{\today}


\maketitle

\abstract{
In several studies it has been observed that, when using stabilised 
$\mathbb{P}_k^{}\times\mathbb{P}_k^{}$ elements for both velocity and pressure, the error for the pressure
is smaller, or even of a higher order in some cases, than the one obtained when using inf-sup 
stable $\mathbb{P}_k^{}\times\mathbb{P}_{k-1}^{}$ (although no formal proof of either of these facts
has been given).  This increase in polynomial order requires the introduction of stabilising terms, since the finite element pairs used do not stability the inf-sup condition. With this motivation, we apply the stabilisation approach to the hybrid discontinuous Galerkin discretisation for the Stokes problem with non-standard boundary conditions.}

\section{Introduction}
\label{sec:introduction}

The interest of this paper is to discretise the Stokes problem with non-standard boundary conditions. 
In~\cite{barrenechea2018hybrid}, a hybrid discontinuous Galerkin (hdG) method was proposed and
analysed for this problem. The finite element method used was the combination of BDM elements of
order $k$ for the velocity, and discontinuous elements of order $k-1$ for the pressure. In this paper
we increase the order of the pressure space to $k$, while keeping the order for the velocity space
fixed as $k$. Since this pair does not satisfy the inf-sup condition, a stabilisation term needs to be added.

The stabilisation term referred to above can be built using a diversity of approaches, but, roughly speaking,
the stabilisation can be residual or non-residual. In~\cite{MR804083} 
the authors added a mesh-dependent term penalising the gradient of the pressure to the formulation. Later, in~\cite{MR868143} this method was restricted and reinterpreted as a Petrov-Galerkin scheme leading to the first consistent stabilised method, and further developments were presented in the works~\cite{MR946377} and~\cite{MR914609}. For a review of different residual stabilised finite element methods for the
Stokes problem, see the review paper~\cite{MR2087327}.

Now, due to their nature, residual methods include unphysical couplings to the formulation, and modify all the entries of the stiffness matrix. Hence, non-residual methods where only a positive semi-definite term penalising
the pressure is added have also being proposed. Examples of this type of methods are the pressure gradient projection~\cite{MR1800154} and local pressure gradient stabilisation~\cite{MR1890352}. 
The methods just mentioned typically use two nested meshes in order to build the method. Thus, to avoid this complication, the local pressure gradient stabilisation has been also presented on the same mesh in~\cite{MR2429873}.
Additionally, methods that use fluctuations of the pressure gradient are not effective when the finite element space
for pressure is the piecewise constant space. The usual way to overcome this is to add pressure jumps
to the formulation, as it has been done, e.g., in~\cite{MR1044204}. These have been shown to be very effective, 
but they do somehow temper with the data structure of the code. To avoid this, the authors in~\cite{MR2079895} present an approach that is based on polynomial-pressure-projection. This method works for low order of polynomials as was shown in~\cite{MR2217373}, and preserves symmetry of the original equation.

In the light of the discussion of the previous paragraphs, in this work we propose a stabilised hdG method
for the Stokes problem with non-standard boundary conditions. The method is reminiscent of the 
Dorhmann-Bochev method (from~\cite{MR2079895}), but uses the same velocity space used in the hdG
method from~\cite{barrenechea2018hybrid}.

\subsection{Notations and model problem}
\label{sec:notation_stabilisation}

Let $\Omega$ be an open polygonal domain in $\mathbb{R}^2$ with Lipschitz boundary $\Gamma := \partial \Omega$. We use boldface font for tensor or vector variables e.g. $\boldsymbol{u}$ is a velocity vector field. The scalar variables will be italic e.g. $p$ denotes pressure scalar value. We define the stress tensor $\boldsymbol{\sigma} := \nu \grad \boldsymbol{u} - p \boldsymbol{I}$ (where $\nu>0$ is the fluid viscosity and $\boldsymbol{I}$ is the
identity matrix) and the flux as $\boldsymbol{\sigma_n} := \boldsymbol{\sigma} \ \boldsymbol{n}$. 
In addition, we denote normal and tangential components as follows ${u_n} := \boldsymbol{u} \cdot \boldsymbol{n}$, $u_t := \boldsymbol{u}  \cdot \boldsymbol{t}$, $\sigma_{nn} := \boldsymbol{\sigma_n} \cdot \boldsymbol{n}$, 
where $\boldsymbol{n}$ is the outward unit normal vector to the
boundary $\Gamma$ and $\boldsymbol{t}$ is a vector tangential to
$\Gamma$ such that $\boldsymbol{n} \cdot \boldsymbol{t} =
0$.

\noindent For $D \subset \Omega$, we use the standard $L^2(D)$ space with the following norm
\begin{eqnarray*}
\|f\|_D^2 := \int_D f^2 \dx & \mbox{for all } f \in L^2(D).
\end{eqnarray*}
Let us define, for $m\in\mathbb{N}$, the following Sobolev spaces
\begin{align*}
H^m(D) & :=  \left\{v \in L^2(D): \ \forall \ |\boldsymbol{\alpha}| \leq m \ \partial^{\boldsymbol{\alpha}} {v} \in L^2(D)\right\} \,, \\
H\left(div, D\right) & := \left\{\boldsymbol{v} \in [L^2(D)]^2: \ \nabla \cdot \boldsymbol{v} \in L^2(D)\right\},
\end{align*} 
where, for $\boldsymbol{\alpha} = (\alpha_1, \alpha_2) \in \mathbb{N}^2$,
$|\boldsymbol{\alpha}| = \alpha_1+\alpha_2$, and $\partial^{\boldsymbol{\alpha}} = \frac{\partial^{|\boldsymbol{\alpha}|}}{\partial x_1^{\alpha_1} \partial x_2^{\alpha_2}}$.
In addition, we will use the standard semi-norm and norm for the Sobolev space $H^m(D)$  
\begin{equation*}
|f|_{H^m(D)}^2 := \sum_{|\boldsymbol{\alpha}| = m} \|\partial^{\boldsymbol{\alpha}} f\|_D^2\quad, \quad \|f\|_{H^m(D)}^2 := \sum_{k = 0}^m |f|_{H^k(D)}^2 \mbox{for all } f \in H^m(D).
\end{equation*}
In this work, we consider the two dimensional Stokes problem with tangential-velocity and normal-flux (TVNF) boundary conditions
\begin{equation}
\label{eq:stokes_TVNF}
\left\{
\begin{array}{rclclr}
 -\nu \Delta \boldsymbol{u} & \ + & \grad p & = & \boldsymbol{f} & \mbox{in } \Omega, \\
 & & \nabla \cdot \boldsymbol{u} & = & 0 & \mbox{in } \Omega, \\
 & & {\sigma_{nn}} & = & {g} & \mbox{ on } \Gamma, \\
 & & u_t & = & 0 & \mbox{ on } \Gamma,
 \end{array}
\right.
\end{equation}
where $\boldsymbol{u}: \bar{\Omega} \rightarrow \mathbb{R}^2$ is the unknown
velocity field, $p:\bar{\Omega} \rightarrow \mathbb{R}$ the pressure, $\nu >
0$ the viscosity, which is considered to be constant, and $\boldsymbol{f} \in [L^2(\Omega)]^2$, ${g} \in L^2(\Gamma)$ are given functions.
The restriction to homogeneous Dirichlet conditions on $u_t$ is made only to simplify the presentation.

Let $\left\{\mathcal{T}_h\right\}_{h > 0}$ be a regular family of triangulations of $\bar{\Omega}$ made of triangles. For each triangulation $\mathcal{T}_h$, $\mathcal{E}_h$ denotes the set of its edges. In addition, for each of element $K \in \mathcal{T}_h$, $h_K := \mbox{diam}(K)$, and we denote $h := \max_{K \in \mathcal{T}_h} h_K$. 
We define following Sobolev spaces on the triangulation $\mathcal{T}_h$ and the set of all edges in $\mathcal{E}_h$
\begin{align*}
L^2(\mathcal{E}_h) & :=  \left\{v: \ v|_E \in L^2(E) \ \forall \ E \in \mathcal{E}_h \right\}, \\
H^m(\mathcal{T}_h) & :=  \left\{v \in L^2(\Omega): \ {v}|_K \in H^m(K) \ \forall \ K \in \mathcal{T}_h \right\} \mbox{ for } m \in \mathbb{N},
\end{align*}
with the corresponding broken norms.

Now we will introduce the finite element spaces that discretise the above spaces. 
Let $k \geq 1$.
We start by introducing the velocity and pressure spaces.
To discretise the velocity $\boldsymbol{u}$ we use the Brezzi-Douglas-Marini space (see~\cite[Section~2.3.1]{MR3097958}) of order $k\ge 1$ defined by
\begin{align*}
 \boldsymbol{BDM_h^k} & := \left\{\boldsymbol{v_h} \in H\left(div, \Omega\right): \ \boldsymbol{v_h}|_K \in \left[\mathbb{P}_k\left(K\right)\right]^2 \ \forall \ {K \in \mathcal{T}_h}\right\}. 
\end{align*}
Associated to this space, we introduce the BDM projection
$\Pi^k: [H^1(\Omega)]^2 \rightarrow \boldsymbol{BDM_h^k}$  defined in~\cite[Section~2.5]{MR3097958}.
The pressure is discretised using the following space 
\begin{align*}
 Q_h^{k} & := \left\{q_h \in L^2\left(\Omega\right): \ q_h|_K \in \mathbb{P}_{k}\left(K\right) \ \forall \ {K \in \mathcal{T}_h}\right\} . 
\end{align*} 
Associated to this space we define the local $L^2(K)$-projection $\Psi_K^{k}: L^2(K) \rightarrow \mathbb{P}_{k}\left(K\right)$ for each $K \in \mathcal{T}_h$ defined as follows. For every $w \in L^2\left(K\right)$, $\Psi^{k}_K(w)$ is the unique element of $\mathbb{P}_{k}\left(K\right)$ satisfying $\int_{K} \Psi^{k}_K(w) v_h dx = \int_{K} {w v_h} dx \quad \forall \ v_h \in \mathbb{P}_{k}\left(K\right),$
and we define the continuous projection $\Psi^k|_K = \Psi^k_K$ for all $K \in \mathcal{T}_h$.

The last ingredient needed in the method described below is a finite element space associated to
a family of Lagrange multipliers associated to the edges of the triangulation. These multipliers will be denoted
by $\tilde{u}$ and are meant to approximate the tangential trace of the velocity $\boldsymbol{u}$ on the
edges of the triangulation. For this, and in order to 
propose a discretisation with fewer degrees of freedom, we discretise the Lagrange multiplier $\tilde{u}$ using the space
\begin{align*}
 M_{h,0}^{k-1} &:= \left\{\tilde{v}_h \in L^2\left(\mathcal{E}_h\right): \ \tilde{v}_h|_E \in \mathbb{P}_{k-1}\left(E \right) \ \forall \ {E \in \mathcal{E}_h}, \ \tilde{v}_h = 0 \mbox{ on } \Gamma\right\}. 
\end{align*} 
 Furthermore, we introduce for all $E \in \mathcal{E}_h$ the $L^2(E)$-projection $\Phi^{k-1}_E: L^2\left(E\right) \rightarrow \mathbb{P}_{k-1}\left(E\right)$ defined as follows. For every $\tilde{w} \in L^2\left(E\right)$, $\Phi^{k-1}_E(\tilde{w})$ is the unique element of $\mathbb{P}_{k-1}\left(E\right)$ satisfying $\int_{E} \Phi^{k-1}_E(\tilde{w}) {\tilde{v}_h} \ds = \int_{E} \tilde{w} {\tilde{v}_h} \ds \quad \forall \ {\tilde{v}_h} \in \mathbb{P}_{k-1}\left(E\right),$
and we denote $\Phi^{k-1}: L^2\left(\mathcal{E}_h\right) \rightarrow M_{h}^{k-1}$ defined as $\Phi^{k-1}|_E := \Phi^{k-1}_E$ for all $E \in \mathcal{E}_h$.


\section{The stabilised method}
\label{sec:TVNF_stabilisation}

Our approach is to write the discrete problem with the same degree of polynomials for velocity and pressure spaces. In other words, denoting $\boldsymbol{V_h} := \boldsymbol{BDM_h^k} \times M_{h,0}^{k-1}$,
 we want to use the space $\boldsymbol{V_h} \times Q^k_h$, instead of $\boldsymbol{V_h} \times Q_h^{k-1}$ as it was done in~\cite{barrenechea2018hybrid}. 
To do this, we need the proper stabilisation term, because this choice of spaces does not guarantee inf-sup stability.

The first ingredient in the definition of the stabilised method for \eqref{eq:stokes_TVNF} we use the same bilinear forms as in~\cite{barrenechea2018hybrid}, this is
\begin{align*}
  & a \left(\left(\boldsymbol{w_h}, {\tilde{w}_h}\right), \left(\boldsymbol{v_h}, {\tilde{v}_h}\right)\right) := \sum_{K \in \mathcal{T}_h} \left(\int_K \nu \grad \boldsymbol{w_h} : \grad \boldsymbol{v_h} \dx \right. \\
&\quad \ - \int_{\partial K} \nu \left(\boldsymbol{\partial_n {w_h}}\right)_t \big(\left(\boldsymbol{v_h}\right)_t - \tilde{v}_h \big) \ds + \varepsilon \int_{\partial K} \nu \big(\left(\boldsymbol{w_h}\right)_t - \tilde{w}_h \big) \left(\boldsymbol{\partial_n {v_h}}\right)_t \ds \\ 
 & \quad  \left. + \nu \frac{\tau}{h_K} \int_{\partial K} \Phi^{k-1}\big(\left(\boldsymbol{w_h}\right)_t - \tilde{w}_h \big) \Phi^{k-1}\big(\left(\boldsymbol{v_h}\right)_t - \tilde{v}_h \big) \ds \right) \\ 
 & b\left(\left(\boldsymbol{v_h}, {\tilde{v}_h}\right), q_h\right)  := - \sum_{K \in \mathcal{T}_h} \int_K q_h \nabla \cdot \boldsymbol{v_h} \dx ,
\end{align*}
where $\varepsilon \in \{-1,1\}$ and $\tau > 0$ is a stabilisation parameter. In addition, to compensate
for the non-inf-sup stability of the finite element spaces we have chosen, we introduce the bilinear form
\begin{equation*}
s\left(p_h, q_h\right) := \frac{1}{{\nu}} \int_{\Omega} \left(p_h - \Psi^{k-1} p_h\right)\left(q_h - \Psi^{k-1} q_h\right) \dx .
\end{equation*}
{With these ingredients we can now present the finite element method analysed in this work:}
\textit{Find $\left(\boldsymbol{u_h}, \tilde{u}_h, p_h\right) \in \boldsymbol{V_h} \times Q^k_h$ such that for all $\left(\boldsymbol{v_h}, \tilde{v}_h, q_h\right) \in \boldsymbol{V_h} \times Q^k_h$}
\begin{equation}
 \label{eq:TVNF_stabilisation_weak_stokes_one_line}
A\left(\left(\boldsymbol{u_h}, \tilde{u}_h,p_h\right),\left(\boldsymbol{v_h}, \tilde{v}_h,q_h\right)\right) = \displaystyle\int_{\Omega}\boldsymbol{f}\boldsymbol{v_h} \dx + \int_{\Gamma} {g} {\left(\boldsymbol{v_h}\right)_n} \ds ,
\end{equation}
where
\begin{align*}
A\left(\left(\boldsymbol{u_h}, \tilde{u}_h,p_h\right),\left(\boldsymbol{v_h}, \tilde{v}_h,q_h\right)\right) := & a \left(\left(\boldsymbol{u_h}, \tilde{u}_h\right),\left(\boldsymbol{v_h}, \tilde{v}_h\right)\right)+ b\left(\left(\boldsymbol{v_h}, \tilde{v}_h\right),p_h\right) \\ 
& + b\left(\left(\boldsymbol{u_h}, \tilde{u}_h\right),q_h\right) - s\left(p_h, q_h\right).
\end{align*}


\subsection{Well-posedness of the discrete problem}
\label{sec:TVNF_stabilisation_stability}
Let us consider the following norm on $\boldsymbol{V_h}$ (see~\cite[Lemma 3.2]{barrenechea2018hybrid} for a proof
that this is actually a norm in $\boldsymbol{V_h}$)
\begin{align}
\label{eq:TVNF_norm}
\nonumber
|||\left(\boldsymbol{w_h}, {\tilde{w}_h}\right)|||^2 := &  \nu \sum_{K \in \mathcal{T}_h} \left(\left|\boldsymbol{w_h}\right|_{H^1(K)}^2 + h_K \left\| \boldsymbol{\partial_n w_h}\right\|_{\partial K}^2  + \frac{\tau}{h_K} \left\|\Phi^{k-1}\big(\left(\boldsymbol{w_h}\right)_t - {\tilde{w}_h}\big)\right\|_{\partial K}^2\right). 
\end{align}

{The first step towards proving the stability of Method~\eqref{eq:TVNF_stabilisation_weak_stokes_one_line} is the following weak inf-sup condition for $b$.}
\begin{lemma}
 \label{l:stabilisation_weak_inf_sup}
There exist constants $C_1, C_2 > 0$, independent of $h_K$ and $\nu$, such that
\begin{equation}
\label{eq:stabilisation_weak_inf_sup}
 \sup_{\left(\boldsymbol{v_h}, \tilde{v}_h\right) \in \boldsymbol{V_h}} \frac{b\left(\left(\boldsymbol{v_h}, \tilde{v}_h\right), q_h\right)}{|||\left(\boldsymbol{v_h}, \tilde{v}_h\right)|||} \geq C_1 \left\|q_h\right\|_{\Omega} - C_2 \left\|q_h - \Psi^{k-1} q_h\right\|_{\Omega} \quad \forall q_h \in Q^k_h.
\end{equation}
\end{lemma}
\begin{proof}
We consider an arbitrary $q_h \in Q^k_h$. Let $\tilde{\Omega}$ be a convex, open, Lipschitz set such that $\Omega \subset \tilde{\Omega}$, and let us consider following extension 
\begin{equation*}
  \hat{q}_h := \left\{
  \begin{array}{lcr}
    q_h & \mbox{in} & \Omega \\
    0 & \mbox{in} & \tilde{\Omega} \setminus \Omega
  \end{array}
  \right. .
\end{equation*}
{Let now $\boldsymbol{\phi}$ be the unique weak solution of the problem}
\begin{equation*}
  \left\{
  \begin{array}{rcll}
      -\Delta \boldsymbol{\phi} & = & \hat{q}_h  & \mbox{on } \tilde{\Omega} \\
      \boldsymbol{\phi} & = & 0 & \mbox{on } \partial \Omega
  \end{array}
  \right. .
\end{equation*}
Since $\tilde{\Omega}$ is convex, then  $\boldsymbol{\phi} \in H^2(\tilde{\Omega})$. Then $\boldsymbol{w} := \grad \boldsymbol{\phi}|_{\Omega}$ belongs to $[H^1(\Omega)]^2$, and for $\tilde{w} := w_t$,
\begin{equation}
\label{eq:weak_inf_sup_equality}
	b\left(\left(\boldsymbol{w}, \tilde{w}\right), q_h\right) = \|q_h\|_{\Omega}^2 \quad \forall q_h \in Q^{k}_h.
\end{equation}
In addition, applying standard  regularity results, see \cite[Section~1.2]{MR3097958}, we get
\begin{equation}
\label{eq:weak_inf_sup_inequality}
	\|\boldsymbol{w}\|_{H^1(\Omega)} \leq \|\grad \boldsymbol{\phi}\|_{H^1(\tilde{\Omega})} \leq c_1 \|q_h\|_{\Omega}.
\end{equation}
In~\cite[Lemma 3.5]{barrenechea2018hybrid} it is shown that there exists a Fortin operator $\boldsymbol{\Pi}: \left[H^1\left(\Omega\right) \right]^2 \rightarrow \boldsymbol{V_h}$ satisfying the following condition: for all $\boldsymbol{v} \in [H^1(\Omega)]^2$ the following holds
\begin{align}
\label{eq:fortin_operator_stokes_first}
 b\left(\left(\boldsymbol{v}, \tilde{v}\right), q_h \right) & = b\left(\boldsymbol{\Pi}\left(\boldsymbol{v}\right), q_h \right) \quad \forall \ {q_h \in Q_h^{k-1}}, \\
\label{eq:fortin_operator_stokes_second}
 |||\boldsymbol{\Pi}\left(\boldsymbol{v}\right)||| & \leq C  \sqrt{\nu}\|\boldsymbol{v}\|_{H^1\left(\Omega\right)}.
\end{align}
Let $\left(\boldsymbol{w_h}, {\tilde{w}_h}\right) := \boldsymbol{\Pi} \left(\boldsymbol{w}\right)$,  
then 
thanks to~\eqref{eq:fortin_operator_stokes_first}, \eqref{eq:weak_inf_sup_equality} and the continuity of $b$ (see~\cite[Lemma 3.3]{barrenechea2018hybrid})
\begin{align*}
	b\left(\left(\boldsymbol{w_h}, \tilde{w}_h\right), q_h\right) & = b\left(\left(\boldsymbol{w}, \tilde{w}\right), q_h\right) - b\left(\left(\boldsymbol{w} - \boldsymbol{w}_h, \tilde{w} - \tilde{w}_h\right), q_h- \Psi^{k-1} q_h\right) \\
& \geq \|q_h\|_{\Omega}^2 - c_2 \sqrt{\sum_{K \in \mathcal{T}_h} \left|\boldsymbol{w_h} - \boldsymbol{w}\right|^2_{H^1(K)}} \left\|q_h - \Psi^{k-1} q_h\right\|_{\Omega}.
\end{align*}
Using the approximation properties of the BDM interpolation operator (see~\cite[Preposition 2.5.1]{MR3097958}) and~\eqref{eq:weak_inf_sup_inequality}
\begin{align*}
	b\left(\left(\boldsymbol{w_h}, \tilde{w}_h\right), q_h\right) & \geq 
\left(\frac{1}{c_1}\|q_h\|_{\Omega} - c_2 c_3 \left\|q_h - \Psi^{k-1} q_h\right\|_{\Omega} \right) \left|\boldsymbol{w}\right|_{H^1(\Omega)} \\
&\ge \left(C_1\|q_h\|_{\Omega} - C_2 \left\|q_h - \Psi^{k-1} q_h\right\|_{\Omega} \right) |||\left(\boldsymbol{w_h}, {\tilde{w}_h}\right)|||\,,
\end{align*}
{where, in  the last estimate we have used the stability of the Fortin operator $\boldsymbol{\Pi}$ 
in the $|||\cdot|||$ norm~\eqref{eq:fortin_operator_stokes_second}. 
This proves the result with
$C_1 = \frac{1}{C \sqrt{\nu} c_1}$ and $C_2 = \frac{c_2 c_3}{C \sqrt{\nu}}$.}
\end{proof}
Before showing an inf-sup condition, we prove the continuity of bilinear form $A$.
\begin{lemma}
 \label{l:stabilisation_continuity}
There exists a constant $C > 0$ such that, for all $\left(\boldsymbol{w_h}, \tilde{w}_h\right), \left(\boldsymbol{v_h}, \tilde{v}_h\right) \in \boldsymbol{V_h}$ and $r_h, q_h \in Q^k_h$, we have
\begin{equation}
\label{eq:stabilisation_continuity}
 \left|A\left(\left(\boldsymbol{w_h}, \tilde{w}_h,r_h\right), \left(\boldsymbol{v_h}, \tilde{v}_h,q_h\right)\right)\right| \leq C |||\left(\boldsymbol{w_h}, \tilde{w}_h,r_h\right)|||_h |||\left(\boldsymbol{v_h}, \tilde{v}_h,q_h\right)|||_h.
\end{equation}
\end{lemma}
\begin{proof}
We use the continuity of the bilinear forms (see~\cite[Lemma 3.3]{barrenechea2018hybrid}) and  the fact that the projection is a bounded operator.
\end{proof}
The final  step towards stability is proving the inf-sup condition for bilinear form $A$.
\begin{lemma}
 \label{l:stabilisation_inf_sup}
There exists $\beta > 0$ independent of $h_K$ such that for all $\left(\boldsymbol{w_h}, \tilde{w}_h,r_h\right) \in \boldsymbol{V_h} \times Q^k_h$ the following holds
	\begin{equation}
\label{eq:stabilisation_inf_sup}
	\sup_{ \left(\boldsymbol{v_h}, \tilde{v}_h,q_h\right) \in \boldsymbol{V_h} \times Q^k_h} \frac{A\left(\left(\boldsymbol{w_h}, \tilde{w}_h,r_h\right),\left(\boldsymbol{v_h}, \tilde{v}_h,q_h\right)\right)}{|||\left(\boldsymbol{v_h}, \tilde{v}_h,q_h\right)|||_h} \geq \beta |||\left(\boldsymbol{w_h}, \tilde{w}_h,r_h\right)|||_h.
	\end{equation}
As a consequence, Problem~\eqref{eq:TVNF_stabilisation_weak_stokes_one_line} is well-posed. 
\end{lemma}
\begin{proof}
Let $\left(\boldsymbol{w_h}, \tilde{w}_h,r_h\right) \in \boldsymbol{V_h} \times Q^k_h$. The idea of the proof is to construct an appropriate $\left(\boldsymbol{v_h}, \tilde{v}_h,q_h\right)$ such that
\begin{equation*}
 A\left(\left(\boldsymbol{w_h}, \tilde{w}_h,r_h\right),\left(\boldsymbol{v_h}, \tilde{v}_h,q_h\right)\right) \geq c |||\left(\boldsymbol{w_h}, \tilde{w}_h,r_h\right)|||_h \ |||\left(\boldsymbol{v_h}, \tilde{v}_h,q_h\right)|||_h.
\end{equation*}
To achieve that we use coercivity of $a$ (see~\cite[Lemma 3.4]{barrenechea2018hybrid}), continuity of $a$ (see~\cite[Lemma 3.3]{barrenechea2018hybrid}) and Lemma~\ref{l:stabilisation_continuity}. For details see~\cite{bosy2017efficient}.
\end{proof}


\subsection{Error analysis}
\label{sec:TVNF_stabilisation_error}

In this section we present the error estimates for the method. 
The addition of the stabilising bilinear form $s(\cdot,\cdot)$ introduced a consistency error. However according to \cite{MR2217373}, this should not be viewed as a serious flaw, as this consistency error can be bounded in an optimal way. The following result is the first step towards that goal.
 \begin{lemma}
\label{l:stabilisation_weak_consistency}
 Let $\left(\boldsymbol{u},p\right) \in \left[H^1\left(\Omega\right) \cap H^2\left(\mathcal{T}_h\right)\right]^2 \times L^2\left(\Omega\right)$ be the solution of the problem \eqref{eq:stokes_TVNF} and $\tilde{u} = u_t$ on all edges of $\mathcal{E}_h$. If $\left(\boldsymbol{u_h}, \tilde{u}_h, p_h \right) \in \boldsymbol{V_h} \times Q^{k}_h$ solves \eqref{eq:TVNF_stabilisation_weak_stokes_one_line}, then for all $\left(\boldsymbol{v_h}, \tilde{v}_h, q_h\right) \in \boldsymbol{V_h} \times Q^{k}_h$ the following holds
\begin{equation}
\label{eq:stabilisation_consistency}
 A\left(\left(\boldsymbol{u} - \boldsymbol{u_h}, \tilde{u} - \tilde{u}_h,p - p_h\right),\left(\boldsymbol{v_h},\tilde{v}_h, q_h\right)\right) = s\left(p, q_h\right)\,. 
 \end{equation}
\end{lemma}

Next, we introduce the following norm
\begin{equation}
\label{eq:TVNF_norm_full}
	|||(\boldsymbol{u}, \tilde{u}, p)|||_h := |||(\boldsymbol{u}, \tilde{u})||| + \frac{1}{\sqrt{\nu}}\|p\|_{\Omega},
\end{equation}
and prove the following  variant of Cea's lemma~\cite[Lemma 2.28]{MR2050138} for this stabilised Stokes problem.
\begin{lemma}
\label{l:stabilisation_Cea}
Let $\left(\boldsymbol{u},p\right) \in \left[H^1\left(\Omega\right) \cap H^2\left(\mathcal{T}_h\right)\right]^2 \times L^2\left(\Omega\right)$ be a solution of~\eqref{eq:stokes_TVNF}, $\tilde{u}=u_t$ on all edges in $\mathcal{E}_h$,
 and $\left(\boldsymbol{u_h}, \tilde{u}_h,p_h\right) \in \boldsymbol{V_h} \times Q^k_h$ solves the discrete problem~\eqref{eq:TVNF_stabilisation_weak_stokes_one_line}. Then there exists $C > 0$, independent of $h$ and $\nu$, such that
\begin{align}
\label{eq:stabilisation_Cea} \nonumber
|||\left(\boldsymbol{u}-\boldsymbol{u_h}, \tilde{u}- \tilde{u}_h,p - p_h\right)|||_h \leq & C \inf_{\left(\boldsymbol{v_h}, \tilde{v}_h,q_h\right) \in \boldsymbol{V_h} \times Q^k_h} |||\left({\boldsymbol{u}}-\boldsymbol{v_h},\tilde{u}- \tilde{v}_h,p - q_h\right)|||_h \\
& +
\frac{C}{\sqrt{\nu}} \left\|p - \Psi^{k-1} p \right\|_{\Omega}.
\end{align} 
\end{lemma}
\begin{proof}
It is a combination of Lemmas~\ref{l:stabilisation_weak_inf_sup},~\ref{l:stabilisation_continuity} and~\ref{l:stabilisation_inf_sup}. For details see~\cite{bosy2017efficient}.
\end{proof}

\begin{lemma}
\label{l:stabilisation_HDG_error_high_order}
Let $\left(\boldsymbol{u},p\right) \in \left[H^1\left(\Omega\right) \cap H^2\left(\mathcal{T}_h\right)\right]^2 \times H^k\left(\Omega\right)$ be a solution of~\eqref{eq:stokes_TVNF}, $\tilde{u}=u_t$ on all edges in $\mathcal{E}_h$, and $\left(\boldsymbol{u_h}, \tilde{u}_h,p_h\right) \in \boldsymbol{V_h} \times Q^k_h$ solves the discrete problem~\eqref{eq:TVNF_stabilisation_weak_stokes_one_line}. 
Then there exists $C > 0$, independent of $h$ and $\nu$, such that
\begin{equation*}
|||\left({\boldsymbol{u}}-\boldsymbol{u_h}, \tilde{u}- \tilde{u}_h, p - p_h\right)|||_h \leq C h^k \left(\sqrt{\nu} \|{\boldsymbol{u}}\|_{H^{k+1}\left(\mathcal{T}_h\right)} + \frac{1}{\sqrt{\nu}} \|p\|_{H^k\left(\mathcal{T}_h\right)} \right).
\end{equation*}
\end{lemma}
\begin{proof}
It is a combination of~\cite[Lemmas 3.8]{barrenechea2018hybrid} and Lemma~\ref{l:stabilisation_Cea} with the local $L^2$-projection approximation~\cite[Theorem 1.103]{MR2050138}.
\end{proof}


\section{Numerical experiments}
\label{sec:hdG_numerics}

The computational domain is the unit square $\Omega = \left(0,1\right)^2$. We present the results for
$k=1$, that is the discrete space is given by $\boldsymbol{BDM_h^1} \times M_{h,0}^0 \times {Q_h^1}$.  We test both the symmetric method
($\varepsilon=-1$) and the non-symmetric method ($\varepsilon=1$). We have followed the recommendation given in \cite[Section~2.5.2]{Lehrenfeldhesis} and taken $\tau =6$. We choose the right hand side $\boldsymbol{f}$ and the boundary condition  $g$ such that the exact solution is given by
\begin{align*}
\boldsymbol{u} = \mbox{curl} \left[\left(1-\cos((1-x)^2)\right)\sin(x^2)\sin(y^2)\left(1-\cos((1-y)^2)\right)\right], &&
p = \tan(xy).
\end{align*}

\begin{figure}[!h]
\centering
    \subfloat[Symmetric bilinear form ($\varepsilon = -1$)\label{fig:Stab_sym_TVNF}]{%
      \includegraphics[width=0.5\textwidth]{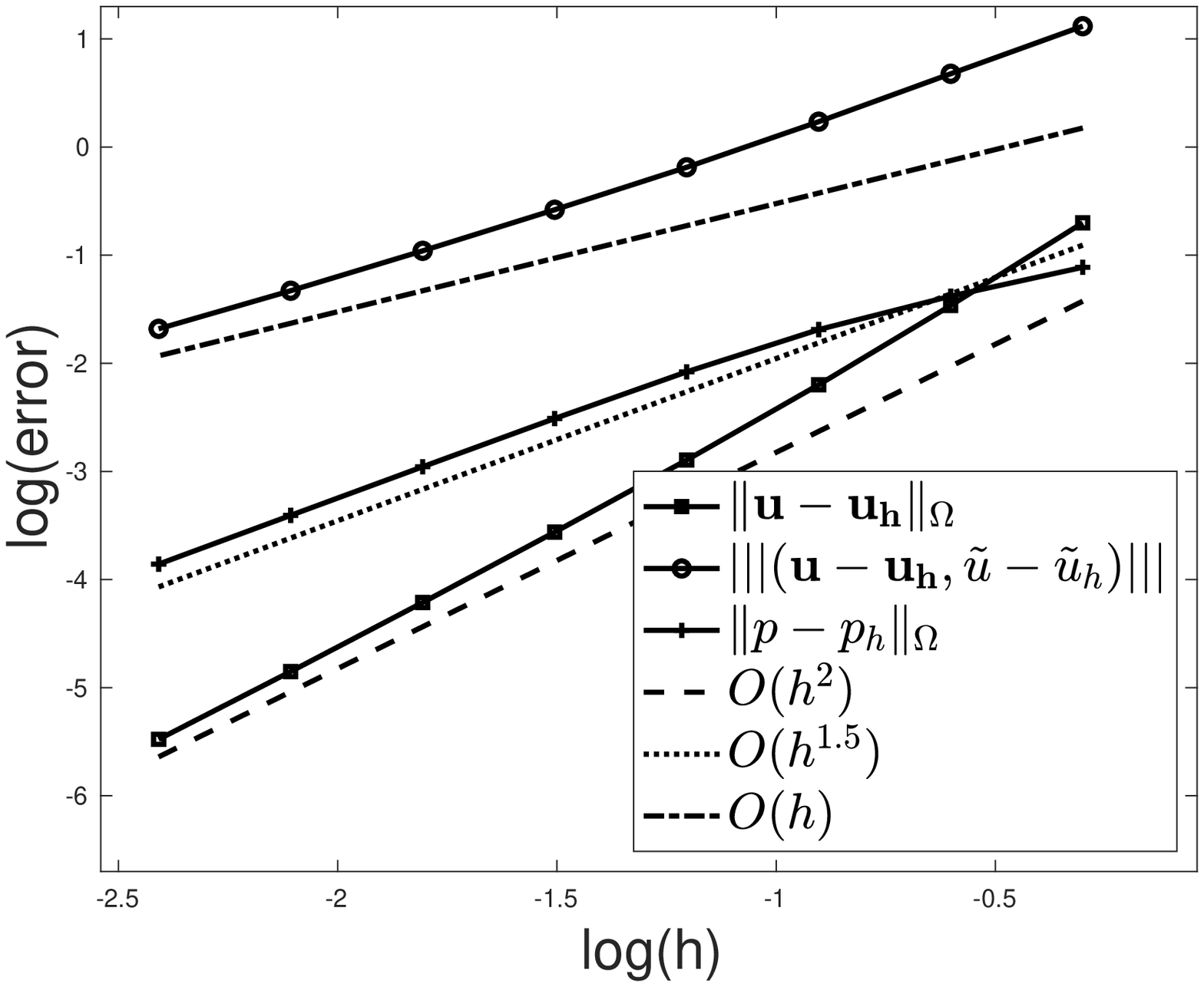}
    }
    \subfloat[Non-symmetric bilinear form ($\varepsilon = 1$)\label{fig:Stab_non_TVNF}]{%
      \includegraphics[width=0.5\textwidth]{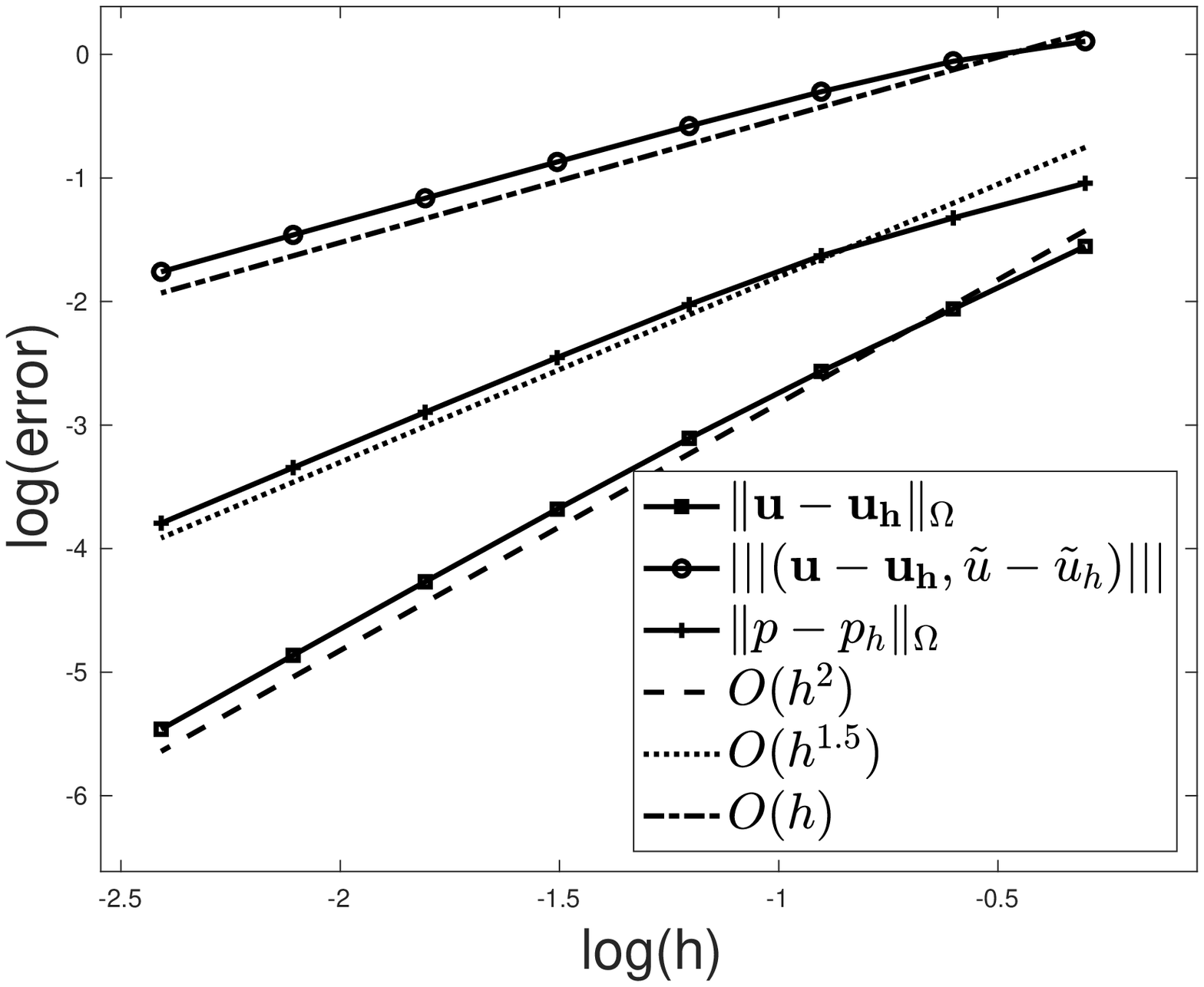}
    }
\caption{Convergence the stabilised method with $k=1$.}
  \end{figure}
In Figure~\ref{fig:Stab_sym_TVNF} and \ref{fig:Stab_non_TVNF} we depict the errors for both the
symmetric and non-symmetric cases, respectively. We can see that they not only validate the theory from Section~\ref{sec:TVNF_stabilisation_error}, but also perform an optimal $h^2$ convergence rate for $\|\boldsymbol{u}-\boldsymbol{u_h}\|_{\Omega}$. Furthermore, we observe an increased order of
convergence for $\|p - p_h\|_{\Omega}$. In fact, the error seems to decrease with $O(h^{3/2})$, rather than the $O(h)$ predicted by the theory.

{\renewcommand\arraystretch{1.1} 
\begin{table}[!h]
\caption{Comparison of the error of the pressure  $||p-p_h||_\Omega$}
\label{tab:Stab_Poiseuille} 
\centering
\begin{adjustbox}{max width=0.85\textwidth}
\begin{tabular}{c | c c | c c}
 & \multicolumn{2}{c | }{\textbf{Symmetric bilinear form (${\varepsilon = -1}$)}} & \multicolumn{2}{c}{\textbf{Non-symmetric bilinear form  (${\varepsilon = 1}$)}} \\
\hline
  $\mathbf{h}$ & $p_h \in Q_h^0$ & $p_h \in Q_h^1$ & $p_h \in Q_h^0$ & $p_h \in Q_h^1$ \\
  \hline
$\mathbf{2^{-1}}$ & 0.152296 & 0.077228 & 0.159019 & 0.090624 \\
$\mathbf{2^{-2}}$ & 0.082775 & 0.041790 & 0.084875 & 0.047488 \\
$\mathbf{2^{-3}}$ & 0.042620 & 0.020500 & 0.043313 & 0.009449 \\
$\mathbf{2^{-4}}$ & 0.021357 & 0.008338 & 0.021513 & 0.003516 \\
$\mathbf{2^{-5}}$ & 0.010676 & 0.003083 & 0.010707 & 0.001269 \\
$\mathbf{2^{-6}}$ & 0.005340 & 0.001105 & 0.005346 & 0.002171 \\
$\mathbf{2^{-7}}$ & 0.002671 & 0.000392 & 0.002672 & 0.000453 \\
$\mathbf{2^{-8}}$ & 0.001336 & 0.000139 & 0.001336 & 0.000161 \\
\end{tabular}
\end{adjustbox}
\end{table}

To stress the last point made in the previous paragraph, in Table~\ref{tab:Stab_Poiseuille} we compare the $L^2$ error of the pressure ($||p-p_h||_\Omega$) for hdG method introduced in~\cite{barrenechea2018hybrid} and stabilised hdG method from Section~\ref{sec:TVNF_stabilisation}. Columns $p_h \in Q_h^0$ are associated with hdG method and $p_h \in Q_h^1$ with stabilised hdG ones. There, we confirm that the pressure error for the stabilised version is
much smaller than the one for the inf-sup stable case, in addition to having an increased order of convergence. 


\section{Conclusion}

In this work we have applied the idea introduced in~\cite{MR2079895} to stabilise the hdG method proposed
in \cite{barrenechea2018hybrid} for the Stokes problem with TVNF boundary conditions. The method adds a simple, symmetric,
term to the formulation, and allowed us to use a higher order pressure space, which, in turn, improved
the pressure convergence (although a proof of this fact is, in general, not available).  This approach was also applied to NVTF boundary conditions (see~\cite{bosy2017efficient}) and can be used for other discontinuous Galerkin methods that deal with Stokes or nearly incompressible elasticity problems. 

Future testing using higher order discretisations is needed to assess whether this approach provides an increase of the convergence rate for the pressure. Thus, the numerical tests with higher order of polynomials for discontinuous finite methods is interest for further research to look for the improvement of the convergence. 

\bibliographystyle{abbrv}
\bibliography{Stabilised_hdG}
\end{document}